\documentclass{amsart}
\usepackage[margin=1.7in]{geometry}
\usepackage{amsfonts,amssymb,amsmath,amsthm}
\usepackage{url}
\usepackage{enumerate}

\usepackage{mathrsfs}
\usepackage{bbold}

\newtheorem{theorem}{Theorem}[section]

\newtheorem{question}[theorem]{Question}
\newtheorem{lemma}[theorem]{Lemma}

\newtheorem{corollary}[theorem]{Corollary}
\theoremstyle{definition}

\newtheorem{definition}[theorem]{Definition}
\numberwithin{equation}{section}

\newcommand{\zfc}{\mathnormal{\mathsf{ZFC}}}
\newcommand{\inter}{\mathop{\bigcap}\limits}
\newcommand{\ch}{\mathnormal{\mathsf{CH}}}

\newcommand{\fs}{\mathop{\mathrm{FS}}}
\newcommand{\fu}{\mathop{\mathrm{FU}}}

\newcommand{\fin}{\mathnormal{\mathsf{FIN}}}

\newcommand{\covm}{\mathnormal{\mathrm{cov}(\mathcal M)}}

\newcommand{\nonn}{\mathnormal{\mathrm{non}(\mathcal N)}}
\newcommand{\covn}{\mathnormal{\mathrm{cov}(\mathcal N)}}

\newcommand{\mm}{\mathnormal{\mathbb{MM}}}
\DeclareMathOperator{\supp}{supp}
\DeclareMathOperator{\cf}{cf}
\DeclareMathOperator{\suc}{succ}
\DeclareMathOperator{\ran}{ran}

\author[D. Fern\'andez]{David~Jos\'e Fern\'andez-Bret\'on}
\address{
Department of Mathematics, University of Michigan \\
2074 East Hall, 530 Church Street \\
Ann Arbor, Michigan, 48109, USA}
\curraddr{
Kurt G\"odel Research Center for Mathematical Logic \\
University of Vienna \\
W\"aringer Stra\ss e 25, 1090 Wien, Austria.
}
\email{david.fernandez-breton@univie.ac.at}
\urladdr{https://homepage.univie.ac.at/david.fernandez-breton/}


\begin{document}

\title{Stable ordered union ultrafilters and $\covm<\mathfrak c$}

\begin{abstract}
A union ultrafilter is an ultrafilter over the finite subsets of $\omega$ that has a base of sets of the form $\fu(X)$, where $X$ is an infinite pairwise disjoint family and $\fu(X)=\{\bigcup F\big|F\in[X]^{<\omega}\setminus\{\varnothing\}\}$. The existence of these ultrafilters is not provable from the $\zfc$ axioms, but is known to follow from the assumption that $\covm=\mathfrak c$. In this article we obtain various models of $\zfc$ that satisfy the existence of union ultrafilters while at the same time $\covm<\mathfrak c$.
\end{abstract}

\maketitle

\section{Introduction}

A union ultrafilter is an ultrafilter that has a basis of $\fu$-sets. Recall that an $\fu$-set (FU stands for ``finite unions'') is a set of the form 
\begin{equation*}
\fu(X)=\left\{\bigcup_{x\in F}x\bigg|F\in[X]^{<\omega}\setminus\{\varnothing\}\right\},
\end{equation*}
where $X\subseteq[\omega]^{<\omega}$ is an infinite pairwise disjoint family of finite subsets of $\omega$. Hindman's finite unions theorem~\cite{hindmanthm} states that, for every finite partition of $[\omega]^{<\omega}$, there exists an infinite pairwise disjoint set $X$ such that $\fu(X)$ is contained within one single element of the partition.\footnote{Hindman's theorem was originally stated in terms of partitioning $\mathbb N$ (rather than $[\omega]^{<\omega}$) and obtaining an infinite set whose finite sums (rather than finite unions) are all in the same cell of the partition. The equivalence between these two versions of the theorem is explained in~\cite[p. 1]{baumgartner-short-proof-of-hindman}.} Thus we can think of union ultrafilters as ultrafilters that contain witnesses for every instance of Hindman's finite unions theorem. In that sense, these ultrafilters arise naturally from Hindman's theorem in much the same way that selective ultrafilters arise naturally from Ramsey's theorem. There are a few variants of union ultrafilters, so we will now proceed to define all of them.

\begin{definition}
We use the symbol $\fin$ to denote the set $[\omega]^{<\omega}\setminus\{\varnothing\}$ of finite subsets of $\omega$. Let $X\subseteq\fin$ and let $u$ be an ultrafilter on $\fin$.
\begin{enumerate}
\item $u$ is said to be a \textit{union ultrafilter} if $(\forall A\in u)(\exists X)(X\text{ is an infinite pairwise }\newline\text{disjoint family}\wedge\fu(X)\in u\wedge\fu(X)\subseteq A)$.
\item $X$ is said to be \textit{ordered}, or a \textit{block sequence}, if $(\forall x,y\in X)(\max(x)<\min(y)\vee\max(y)<\min(x))$. On occasion we will say that a sequence $\vec{x}=\langle x_n\big|n<\omega\rangle$ is a block sequence; by this we imply the assumption that $(\forall n<\omega)(\max(x_n)<\min(x_{n+1}))$.
\item The union ultrafilter $u$ is said to be \textit{ordered} if, for every $A \in u$, there exists an infinite block sequence $X$ such that $u\ni\fu(X)\subseteq A$.
\item The union ultrafilter $u$ is said to be \textit{stable} if, for every countable family $\{A_n\big|n<\omega\}\subseteq u$, there exists an infinite pairwise disjoint family $X$ such that, for all $n<\omega$, there is a finite $F$ such that $\fu(X\setminus F)\subseteq A_n$.
\end{enumerate}
\end{definition}

These ultrafilters have many desirable properties from the perspective of algebra in the \v{Cech}--Stone compactification. Recall that the \v{C}ech--Stone compactification $\beta\fin$ is the set of all ultrafilters over $\fin$. This is a compact Hausdorff space, when equipped with the topology that has as basic open sets all sets of the form
\begin{equation*}
\overline{A}=\{u\in\beta\fin\big|A\in u\},
\end{equation*}
where $A\subseteq\fin$; these basic open subsets turn out to be, in fact, clopen. Furthermore, this compact Hausdorff space contains a dense copy of $\fin$ embedded as a discrete subspace, via the embedding $x\longmapsto\{A\subseteq\fin\big|x\in A\}$. There is a correspondence between closed subsets of $\beta\fin$ and filters on $\fin$, namely, if $\mathscr F$ is a filter, we can associate to it the set $\{u\in\beta\fin\big|\mathscr F\subseteq u\}$ of ultrafilters extending it, and this set will be closed; conversely, if $F\subseteq\beta\fin$ is a closed set, then $\mathscr F=\bigcap F$ will be a filter; these two correspondences are inverses of each other. The reader interested in the minutiae of the \v{C}ech--Stone compactification should see~\cite{hindmanstrauss}.

A set $A\subseteq\fin$ will be called {\it min-unbounded} if it contains elements with arbitrarily large minima; a filter $\mathscr F$ on $\fin$ will be called {\it min-unbounded} if all of its elements are min-unbounded. Notice that a filter is min-unbounded if and only if it extends the filter generated by the family $\{\{x\in\fin\big|n<\min(x)\}\big|n<\omega\}$; as a consequence of that, the set of all min-unbounded ultrafilters, denoted by $\gamma\fin$, will be a closed (hence compact) subspace of $\beta\fin$. This subspace can be equipped with a semigroup operation, which we will denote by $+$, and which is given by
\begin{equation*}
u+v=\{A\subseteq\fin\big|\{x\in\fin\big|\{y\in\fin\big|\max(x)<\min(y)\wedge x\cup y\in A\}\in v\}\in u\}
\end{equation*}
(the set $\{y\in\fin\big|\max(x)<\min(y)\wedge x\cup y\in A\}$ will be usually abbreviated by writing $-x+A$). This operation turns $\gamma\fin$ into a right-topological semigroup (this means that, for every $v\in\gamma\fin$, the right translation $u\longmapsto u+v$ is a continuous function). Recall that the Ellis--Numakura Lemma (\cite{ellis-lemma,numakura-lemma}; see~\cite[Theorem 2.5]{hindmanstrauss}) states that every compact right-topological semigroup has idempotent elements (i.e., elements satisfying $u+u=u$). Thus, $\gamma\fin$, as well as all of its closed subsemigroups, must have idempotent elements. Idempotent ultrafilters are very useful in providing a beautiful proof of Hindman's finite-unions theorem (the Galvin--Glazer proof, which can be found in~\cite[Theorem 5.8, Corollaries 5.9 and 5.10]{hindmanstrauss}); on the flip side, union ultrafilters --ultrafilters that arise naturally from the aforementioned theorem-- turn out to be idempotent elements of $\gamma\fin$. Union ultrafilters have a significant amount of nice properties. We have already mentioned their idempotence; it is worth mentioning that union ultrafilters also satisfy the so-called \textit{trivial sums property}. This means that they can only be decomposed as a sum of two ultrafilters in the trivial way: whenever $u=v+w$, for $v,w\in\gamma\fin$, it must be the case that $v=w=u$.

The additional properties of being ordered and stable provide us with ultrafilters that witness a higher dimensional version of Hindman's theorem, namely, the so-called Milliken--Taylor theorem. That is, a union ultrafilter $u$ will be stable and ordered if and only if for every colouring $c:[\fin]^2\longrightarrow 2$, there exists a block sequence $X$ such that $\fu(X)\in u$ and the set
\begin{equation*}
[\fu(X)]^2_<=\{(x,y)\in\fu(X)\times\fu(X)\big|\max(x)<\min(y)\vee\max(y)<\min(x)\}
\end{equation*}
is $c$-monochromatic; see~\cite[Theorem 4.2]{blassunion}. For this reason, some authors use the terminology of {\it Milliken--Taylor ultrafilters} to refer to stable ordered union ultrafilters; this was the terminology first used by Matet~\cite{matet1,matet2}, who introduced these ultrafilters independently from and roughly at the same time as Blass. For a more modern example of using this terminology, see~\cite{heike-nearcoherencespectrum}.

The existence of union ultrafilters is known to be both consistent and independent of the $\zfc$ axioms. Independence was established by Blass and Hindman~\cite[Theorem 3]{blasshindman}, who showed that the existence of a union ultrafilter $u$ implies the existence of two, not Rudin--Keisler equivalent, P-points (namely the Rudin--Keisler images $\max(u)$ and $\min(u)$ of $u$ under the mappings $\max:\fin\longrightarrow\omega$ and $\min:\fin\longrightarrow\omega$, respectively). On the other hand, the consistency of the existence of union ultrafilters was established inadvertently by Hindman~\cite[Theorem 3.3]{hindman}, who showed that, if we assume both $\ch$ and Hindman's theorem (which was still not known to be true at the time), then there exist union ultrafilters. Nowadays, we know, thanks to the work of Eisworth~\cite{eisworth-union}, that in fact it suffices to assume that $\covm=\mathfrak c$ to ensure the existence of union ultrafilters. Recall that the cardinal characteristic $\covm$ is defined to be the least cardinality of a family of meagre sets whose union completely covers the real line; equivalently, $\covm$ is the least cardinality of a family of dense subsets of the partial order corresponding to Cohen's forcing for which there is no filter meeting them all. Eisworth proved, furthermore, that the cardinal characteristic equality $\covm=\mathfrak c$ is equivalent to what one might naturally call \textit{generic existence} of stable ordered union ultrafilters, that is, the statement that every filter generated by $<\!\mathfrak c$ sets of the form $\fu(X)$ can be extended to a stable ordered union ultrafilter~\cite[Theorem 9]{eisworth-union}.

In this article we explore models of $\zfc$ that satisfy the existence of union ultrafilters together with the strict inequality $\covm<\mathfrak c$. A few such models have already been obtained, such as some of the ones considered in~\cite{heike-nearcoherencespectrum};\footnote{In these models, moreover, a significant amount of precise information can be obtained about the number of non-isomorphic Ramsey ultrafilters that populate the model.} another notable example is Sacks model obtained by iterating perfect set forcing with countable supports~\cite{conmichael-diamond-union}. In this article we exhibit various other ways to obtain such models. In Section 2 we provide a family of such models obtained by means of finite support iteration of ccc forcing notions. In Section 3 we provide yet another model, this time obtained by means of a countable support iteration of proper forcing notions. Finally, in Section 4 we provide a model of set theory where a single-step forcing adjoining a generic stable ordered union ultrafilter produces a generic extension satisfying $\covm<\mathfrak c$ (an iterated forcing argument will be needed, however, to obtain an adequate ground model in which to perform this single-step forcing).

\section{Cofinally many Cohen reals}

In this section we will see that any iteration of forcings, of uncountable regular length, which adds Cohen reals cofinally often, must contain a stable ordered union ultrafilter. We will take advantage of this to define some specific such iterations that force the generic extension to satisfy $\covm<\mathfrak c$. The following lemma, which will be needed later, is well known if one replaces ``subbasis'' with ``basis'', but currently lacks (to the best of my knowledge) a published written proof in the form stated below. Recall that a subset $\mathscr S$ of a filter $\mathscr F$ is said to be a {\it subbasis} if its closure under finite intersections forms a basis for $\mathscr F$, in other words, if $\mathscr F=\{X\big|(\exists F\in[\mathscr S]^{<\omega})(\cap_{Y\in F}Y\subseteq X)\}$.

\begin{lemma}\label{extendtoidempotent}
Let $\mathscr F$ be a min-unbounded filter on $\fin$ with a subbasis $\mathscr S$ consisting of $\fu$-sets. Then, $\mathscr F$ can be extended to an idempotent ultrafilter $u\in\gamma\fin$.
\end{lemma}

\begin{proof}
We will show that the closed subset $S$ of $\gamma\fin$ consisting of ultrafilters that extend $\mathscr F$ must be a subsemigroup of $\gamma\fin$. This will allow us to conclude that $S$ is a compact right-topological semigroup, and so it must contain some idempotent element $u$ by the Ellis--Numakura lemma; then by definition $\mathscr F\subseteq u$. To see that $S$ is indeed a subsemigroup, let $v,w\in S$ be two ultrafilters. We want to show that $v+w\in S$, in other words, that $\mathscr F\subseteq v+w$; since $\mathscr S$ is a subbasis for $\mathscr F$, this will be the case if and only if $\mathscr S\subseteq v+w$. So take an arbitrary element of $\mathscr S$; by assumption, such an element is of the form $\fu(X)$ for $X\in[\fin]^\omega$ an infinite pairwise disjoint family. Notice that, since $\mathscr F$ is min-unbounded, then $\fu(X\setminus F)\in\mathscr F$ for any finite $F$; in particular, if $x\in\fu(X)$, say that $x=\bigcup_{i=1}^n x_n$ for $x_n\in X$, then $\fu(X\setminus\{x_1,\ldots,x_n\})\in\mathscr F\subseteq w$. Since any $y\in\fu(X\setminus\{x_1,\ldots,x_n\})$ is such that $x\cup y\in\fu(X)$, we conclude that $-x+\fu(X)\supseteq\fu(X\setminus\{x_1,\ldots,x_n\})\in w$ for all $x\in\fu(X)\in\mathscr S\subseteq v$. Therefore
\begin{equation*}
\{x\in\fin\big|-x+\fu(X)\in w\}\supseteq\fu(X)\in v,
\end{equation*}
showing that $\fu(X)\in v+w$. The proof is finished.
\end{proof}

By looking at the definition of ultrafilter addition, one can see that, if $u$ is idempotent and $A\in u$, then $A^\star=\{x\in A\big|-x+A\in u\}\in u$. Furthermore, for each $x\in A^\star$, it is the case that $-x+A^\star\in u$; in other words, $(A^\star)^\star=A^\star$~\cite[Lemma 4.14]{hindmanstrauss}. We are now ready to state the general result about union ultrafilters on extensions that have cofinally many Cohen reals; this result is reminiscent of analogous results of Roitman~\cite[Theorems 1 and 9]{roitman-cofinallycohen} regarding P-points and selective ultrafilters.

\begin{theorem}\label{shortfs}
Let $\langle\mathbb P_\alpha,\mathring{\mathbb Q_\alpha}\big|\alpha\leq\kappa\rangle$, where $\kappa$ has uncountable cofinality, be a forcing iteration that either:
\begin{enumerate}
\item has finite supports, or
\item satisfies that, for cofinally many $\alpha<\kappa$, $\mathbb P_\alpha\vDash``\mathring{\mathbb Q_\alpha}\text{ adds a Cohen real}"$.
\end{enumerate}
Then, if $G$ is a $\mathbb P_\kappa$-generic filter, we have that $\mathbf{V}[G]\vDash\text{There exists a stable ordered}\newline\text{ union ultrafilter}$.
\end{theorem}

\begin{proof}
By merging together iterands along a cofinal sequence if necessary, we may assume in fact that $\kappa$ is an uncountable regular cardinal, and furthermore that each $\mathring{\mathbb Q_\alpha}$ adds at least one Cohen real (in the case of finite support iterations, we just need to do the merging in such a way that each of the new $\mathring{\mathbb Q_\alpha}$ includes a limit stage of the original iteration). Given a $\mathbb P_\kappa$-generic filter $G$, we let $G_\alpha=G\cap\mathbb P_\alpha$ and $V_\alpha=\mathbf{V}[G_\alpha]$. Let $\langle c_\alpha\big|\alpha<\kappa\rangle$ be a sequence of Cohen reals added by the iterands, in other words, for every $\alpha<\kappa$, we will have that $c_\alpha\in V_{\alpha+1}\setminus V_\alpha$ is Cohen over $V_\alpha$. We recursively define an increasing sequence of filters $\mathscr F_\alpha$ on $\fin$ satisfying the following five conditions for every $\alpha\leq\kappa$:
\begin{enumerate}
\item $\mathscr F_\alpha\in V_\alpha$,
\item $\mathscr F_\alpha$ has a subbasis of sets of the form $\fu(X)$, for $X$ an infinite block sequence,
\item $\alpha=\bigcup\alpha\Rightarrow\mathscr F_\alpha=\bigcup_{\xi<\alpha}\mathscr F_\xi$,
\item for every $A\in V_\alpha\cap\wp(\fin)$, there exists a block sequence $X$ such that either $\fu(X)\subseteq A$ or $\fu(X)\subseteq\fin\setminus A$ and $\fu(X)\in\mathscr F_{\alpha+1}$,
\item for every countable sequence $\{A_n\big|n<\omega\}\subseteq\mathscr F_\alpha$, there exists a block sequence $X$ such that $(\forall n<\omega)(X\subseteq^* A_n)$ and with $\fu(X)\in\mathscr F_{\alpha+1}$.
\end{enumerate}
If we are successful in choosing such $\mathscr F_\alpha$, in the end we will obtain the filter $u=\mathscr F_\kappa=\bigcup_{\alpha<\kappa}\mathscr F_\alpha\in V_\kappa=\mathbf{V}[G]$. Notice that, since $\kappa$ has uncountable cofinality, every $A\in\wp(\fin)$, as well as every countable sequence $\{A_n\big|n<\omega\}$, that belong to $\mathbf{V}[G]$, will already belong to some $V_\alpha$ for $\alpha<\kappa$; this will imply that $u$ is an ordered union ultrafilter by requirement (4), and moreover $u$ will be stable by requirement (5).

Hence we now focus on building the $\mathscr F_\alpha$ satisfying all four requirements. Abiding by requirement (3) means that we will not worry at limit stages, so let us focus on the case where we know $\mathscr F_\alpha$, and we attempt to define $\mathscr F_{\alpha+1}$. Let us start working within $V_\alpha$. By the inductive requirement (2) on $\mathscr F_\alpha$, along with Lemma~\ref{extendtoidempotent}, there exists an idempotent ultrafilter $u_\alpha\in V_\alpha$ extending $\mathscr F_\alpha$. Let us now move to $V_{\alpha+1}=\mathbf{V}[G_{\alpha+1}]$, and recall that $c_\alpha\in V_{\alpha+1}$ is a real that is Cohen over $V_\alpha$; let us think of $c_\alpha$ as a function $c_\alpha:\fin\longrightarrow 2$. For the purpose of taking minima in $\fin$, well-order it (in order-type $\omega$) by declaring that $x<y$ iff $\max(x\bigtriangleup y)\in y$. For each $A\in u_\alpha$ we define a block sequence $\vec{x^A}=\langle x_n^A\big|n<\omega\rangle$ such that $\fu(\ran(\vec{x^A}))\subseteq A$, by means of the following recursion:
\begin{equation*}
x_0^A=\min\{x\in\fin\big|c_\alpha(x)=1\wedge x\in A^\star\},
\end{equation*}
and
\begin{eqnarray*}
x_{n+1}^A & = & \min\left\{x\in\fin\bigg|\max(x_n^A)<\min(x)\wedge\right. \\
 & & \left.c_\alpha(x)=1\wedge x\in A^\star\cap\left(\bigcap_{y\in\fu(\{x_0^A,\ldots,x_n^A\})}(-y+A)^\star\right)\right\}.
\end{eqnarray*}
The usual Galvin--Glazer argument for proving Hindman's theorem using idempotent ultrafilters can be used here to see that, as long as each $x_n^A$ is defined, it will be indeed the case that $\fu(\{x_n^A\big|n<\omega\})\subseteq A$. Using the idempotence of $u_\alpha$ and the fact that $A\in u_\alpha$, it is easy to show by induction on $n<\omega$ that the set $A^\star\cap\left(\bigcap_{y\in\fu(\{x_0^A,\ldots,x_n^A\})}(-y+A^\star)\right)\in u_\alpha$, and in particular it is infinite and min-unbounded. Thus the set of all $f:\fin\longrightarrow 2$ such that $f^{-1}[\{1\}]$ intersects that set in a min-unbounded set is comeagre in $2^\fin$, and therefore, since $c_\alpha$ is Cohen over $V_\alpha$, the above recursion really defines an infinite block sequence $\vec{x^A}$. Furthermore, for each countable sequence $\mathscr A=\{A_n\big|n<\omega\}\subseteq u_\alpha$ such that $\mathscr A\in V_\alpha$, recursively define the block sequence $\vec{x^{\mathscr A}}=\langle x_n^{\mathscr A}\big|n<\omega\rangle$ by:
\begin{equation*}
x_0^{\mathscr A}=\min\{x\in\fin\big|c_\alpha(x)=1\wedge x\in A_0^\star\},
\end{equation*}
and
\begin{eqnarray*}
x_{n+1}^{\mathscr A} & = & \min\left\{x\in\fin\bigg|\max(x_n^{\mathscr A})<\min(x)\wedge\right. \\
 & & \left. c_\alpha(x)=1\wedge x\in\bigcap_{k\leq n}\left(A_k^\star\cap\left(\bigcap_{y\in\fu(\{x_k^A,\ldots,x_n^A\})}(-y+A_k)^\star\right)\right)\right\}.
\end{eqnarray*}
Notice that, in case the sequence $\langle x_n^{\mathscr A}\big|n<\omega\rangle$ is infinite, then (again by the usual Galvin--Glazer argument, since after stage $n$ we start taking care of $A_n$) it will be the case that $(\forall n<\omega)(\ran(\vec{x^{\mathscr A}})\subseteq^* A_n)$ and, in fact, for every $n<\omega$ we will have that $\fu(\{x_k^{\mathscr A}\big|k>n\})\subseteq A_n$. For each $n<\omega$ we have\newline $\bigcap_{k\leq n}\left(A_k^\star\cap\left(\bigcap_{y\in\fu(\{x_k^A,\ldots,x_n^A\})}(-y+A_k)^\star\right)\right)\in u_\alpha$, so in particular this set is infinite and min-unbounded, and therefore the set of all $f:\fin\longrightarrow 2$ such that this set intersects $f^{-1}[\{1\}]$ in a min-unbounded set is comeagre in $2^\fin$. Hence, since $c_\alpha$ is Cohen over $V_\alpha$, we can conclude that the above definition yields an infinite block sequence $\langle x_n^{\mathscr A}\big|n<\omega\rangle$.

Let us now show that the family
\begin{equation*}
\mathscr B_\alpha=\{\ran(\vec{x^A})\big|A\in u_\alpha\}\cup\{\ran(\vec{x^{\mathscr A}})\big|\mathscr A=\{A_n\big|n<\omega\}\subseteq u_\alpha\wedge\mathscr A\in V_\alpha\}
\end{equation*}
is centred. In fact, if $A_1,\ldots,A_n\in u_\alpha$ and $\mathscr A_1,\ldots,\mathscr A_m\in[u_\alpha]^{<\omega}\cap V_\alpha$, notice that, for every choice of $k_1,\ldots,k_n,l_1,\ldots,l_m<\omega$, the set
\begin{eqnarray*}
 & & \left(\bigcap_{i=1}^n\left(A_i^\star\cap\left(\bigcap_{y\in\fu(\{x_0^{A_i},\ldots,x_{k_i-1}^{A_i}\})}(-y+A_i^\star)\right)\right)\right) \\  & & \cap\left(\bigcap_{i=1}^m\left(\bigcap_{k=1}^{l_i-1}\left((A_j^i)^\star\cap\left(\bigcap_{y\in\fu(\{x_k^{A_j^i},\ldots,x_{l_i-1}^{A_j^i}\})}(-y+A_j^i)^\star\right)\right)\right)\right)
\end{eqnarray*}
(where $\mathscr A_i=\{A_j^i\big|j<\omega\}$) belongs to $u_\alpha$ and hence it is min-unbounded, and therefore the set of all $f:\fin\longrightarrow 2$ such that this set intersects $f^{-1}[\{1\}]$ in a min-unbounded set is comeagre in $2^\fin$. Thus, since $c_\alpha$ is Cohen over $V_\alpha$, we conclude that
\begin{equation*}
\left(\bigcap_{i=1}^n\{x_j^{A_i}\big|j<\omega\}\right)\cap\left(\bigcap_{i=1}^m\{x_j^{\mathscr A_i}\big|j<\omega\}\right)
\end{equation*}
is nonempty, and in fact, it is infinite and min-unbounded (by Cohen-genericity of $c_\alpha$ over $V_\alpha$).

Therefore $\mathscr B_\alpha$ is a centred family; this immediately implies (since $X\subseteq\fu(X)$ for every block sequence $X$) that the family $\mathscr S_\alpha=\{\fu(X)\big|X\in\mathscr B_\alpha\}$ is centred as well; we let $\mathscr F_{\alpha+1}$ be the filter generated by $\mathscr S_\alpha$ (that is, $\mathscr F_{\alpha+1}$ has $\mathscr S_\alpha$ as a subbasis). Hence the inductive requirement (2) holds of $\mathscr F_{\alpha+1}$ (and so does inductive requirement (1), since the whole construction was carried out in $V_{\alpha+1}$). For every $\mathscr A=\{A_n\big|n<\omega\}\subseteq\mathscr F_\alpha\subseteq u_\alpha$, we know that $\fu(\ran(\vec{x^{\mathscr A}}))\in\mathscr F_{\alpha+1}$, which takes care of requirement (5) for $\mathscr F_{\alpha+1}$. Furthermore, if $A\in V_\alpha\cap\wp(\fin)$, then either $A$ or $\fin\setminus A$ belongs to $u_\alpha$; in either case, the presence of $\fu(\ran(\vec{x^A}))$ (respectively $\fu(\ran(\vec{x}^{\fin\setminus A}))$) in $\mathscr F_{\alpha+1}$ verifies requirement (4). We have shown that our filter $\mathscr F_{\alpha+1}$ still satisfies all five inductive requirements (requirement (3) is vacuously true at this stage). Hence the induction can continue, and we are done.
\end{proof}

\begin{corollary}\label{fsmodel}
Let $\mathbf{V}$ be a model of $\zfc$ satisfying $\mathbf{V}\vDash\cf(\kappa)=\kappa<\lambda=\mathfrak c$. Suppose we perform a finite support iteration of c.c.c. forcings of size $\leq\!\!\lambda$. If the iteration adds a dominating real cofinally often, then in the generic extension it will be the case that $\covm=\kappa<\lambda=\mathfrak c$ and there are stable ordered union ultrafilters.
\end{corollary}

\begin{proof}
Suppose that our iteration is $\langle\mathbb P_\alpha,\mathring{\mathbb Q_\alpha}\big|\alpha<\kappa\rangle$. One of our assumptions is that $\mathbb P_\alpha\Vdash``|\mathring{Q_\alpha}|\leq\check{\lambda}"$, so we may assume that, in fact,
\begin{equation*}
\mathbb P_\alpha\Vdash``\text{The underlying set of }\mathring{Q_\alpha}\text{ is a subset of }\check{\lambda}".
\end{equation*}
Since $\lambda^\omega=(2^\omega)^\omega=2^\omega=\lambda$ (in $\mathbf{V}$) and each of the iterands is c.c.c., a straightforward computation on the number of (nice) names for reals and (nice) names for elements of $\lambda$ allows us to show, by induction on $\alpha$, that each $\mathbb P_\alpha$ has a dense subset of size $\lambda$ and that $\mathbb P_\alpha\Vdash``\mathfrak c=\check{\lambda}"$. Now, notice that
\begin{equation*}
\mathbf{V}[G]\vDash(\exists u)(u\text{ is a stable ordered union ultrafilter}),
\end{equation*}
by Theorem~\ref{shortfs}. On the other hand, we claim that
\begin{equation*}
\mathbf{V}[G]\vDash\covm=\kappa.
\end{equation*}
To see this, notice that, on the one hand, a Cohen real is added at each limit stage of the iteration. A real is Cohen over $\mathbf{V}[G_\alpha]$ if and only if it does not belong to any Borel meagre set coded in $\mathbf{V}[G_\alpha]$, where $G_\alpha=G\cap\mathbb P_\alpha$. Since any family of $<\!\!\kappa$ Borel meagre subsets of $\mathbb R$ must be completely contained within some stage $\alpha<\kappa$ of the iteration, and some Cohen real gets added by this iteration after stage $\alpha$, such a family cannot cover $\mathbb R$. Hence $\covm\geq\kappa$. On the other hand, by merging iterands if necessary, we may assume that $\mathbb P_\alpha\Vdash``\mathring{\mathbb Q_\alpha}\text{ adds a dominating real}"$ for all $\alpha<\kappa$; let $x_\alpha$ be a dominating real added by $\mathring{\mathbb Q_\alpha}$. Then the family $\{x_\alpha\big|\alpha<\kappa\}\in\mathbf{V}[G]$ will be a dominating family, showing that $\mathfrak d\leq\kappa$ in $\mathbf{V}[G]$. Since (it is provable in $\zfc$ that) $\covm\leq\mathfrak d$, it follows that $\covm\leq\kappa$ in $\mathbf{V}[G]$. This shows that $\covm=\kappa<\lambda=\mathfrak c$ in $\mathbf{V}[G]$. With this, the proof is complete.
\end{proof}

\section{Countable support iterations}

In this section we show that it is also possible to obtain models of $\zfc$, with stable ordered union ultrafilters and $\covm<\mathfrak c$, using a countable support iteration of proper forcing notions. The main forcing notion for our construction is the following.

\begin{definition}
 Given an ultrafilter $u$ on $\fin$, we define the {\it ultraLaver forcing on $u$} to be the partially 
 ordered set $\mathbb L(u)$ whose elements are subtrees $p$ of the full tree $\fin^{<\omega}$ (that is, we require that $p$ is closed under initial segments) with a stem $s_p\in p$ satisfying that every node $t\in p$ is comparable with $s_p$, such that ``the branching of $p$ is in $u$ above $s_p$'', meaning that for every $t\in p$ such that $t\geq s_p$, the set of immediate successors $\suc_p(t)=\{x\in\fin\big|t\frown\langle x\rangle\in p\}\in u$. The ordering is given by $p\leq q$ iff $p\subseteq q$.
\end{definition}

Normally this forcing notion is defined on $\omega$ rather than $\fin$, and with a different ultrafilter for each node, but the definition as stated above is all we need for our purposes. The following are well-known properties of ultraLaver forcing (see for example \cite[Section 1A]{borelanddualborel}), and are also not terribly difficult to prove.

\begin{itemize}
 \item $\mathbb L(u)$ is $\sigma$-centred (hence c.c.c. and proper).
 \item $\mathbb L(u)$ has the {\it pure decision property}: given any statement $\varphi$ 
 in the forcing language and any condition $p\in\mathbb L(u)$, it is possible to find a 
 \textit{pure extension} $p'\leq^0 p$ (that is, $p'\leq p$ and $s_{p'}=s_p$) deciding 
 $\varphi$ (i.e., either $p'\Vdash\varphi$ or $p'\Vdash\neg\varphi$).
 \item As a direct consequence of the previous point, whenever $F$ is a (ground-model) finite set, and 
 $\mathring{x}$ is an $\mathbb L(u)$-name such that some condition $p$ forces $p\Vdash``\mathring{x}\in\check{F}"$, there is a pure extension $p'\leq^0 p$ and an element $y\in F$ such that $p'\Vdash``\mathring{x}=\check{y}"$.
\end{itemize}

Note that, at this point, we are still not assuming any special property of $u$ other than its being an ultrafilter. The following lemma shows that the situation becomes quite interesting when $u\in\gamma\fin$ is an idempotent ultrafilter. In what follows, we will denote by $\mathring{X}$ the $\mathbb L(u)$-name for the generic subset that arises from the generic filter (which is the union of all the stems of, or equivalently the intersection of all conditions in, the generic filter); consequently, the generic extension will be denoted by $\mathbf{V}[X]$. It is easy to see that, if the ultrafilter $u$ is min-unbounded, then every condition $p\in\mathbb L(u)$ will force $p\Vdash``\mathring{X}\setminus\ran(\check{s_p})\text{ is a block sequence}"$.

\begin{lemma}\label{laverdiagonalizes}
Let $u\in\gamma\fin$ be an idempotent ultrafilter. In the generic extension $\mathbf{V}[X]$ obtained by forcing with $\mathbb L(u)$, for every (ground model) set $A\in u$ there is a finite set $F$ such that $\fu(X\setminus F)\subseteq A$.
\end{lemma}

\begin{proof}
It suffices to prove that every condition $p$ can be extended to a condition $q$ that forces $``\fu(\mathring{X}\setminus\ran(\check{s_q}))\subseteq\check{A}"$. In order to do that, we recursively define the levels $(q)_n$ of the condition $q\leq p$; in doing so, we will essentially reproduce the Galvin--Glazer argument for proving Hindman's theorem along the tree $p$. First we let $s_q=s_p$. Now suppose that we have defined the $n$-th level of $q$ above the stem, $(q)_{|s_p|+n}$; furthermore, assume that, for each $t=s_p\frown\langle x_0,\ldots,x_{n-1}\rangle\in (q)_{|s(p)|+n}$, it is the case that $\fu(\{x_0,\ldots,x_{n-1}\})\subseteq A^\star$. Then we will define $(q)_{|s_p|+n+1}$ by letting, for every $t=s_p\frown\langle x_0,\ldots,x_{n-1}\rangle\in (q)_{|s_p|+n}$, its set of successors be
\begin{equation*}
\suc_q(t)=\suc_p(t)\cap A^\star\cap\left(\inter_{x\in\fu(\{x_0,\ldots,x_{n-1}\})}-x+A^\star\right)
\end{equation*}
and notice that the inductive hypothesis still holds for all of the new nodes $t\frown\langle x\rangle\in(q)_{|s_p|+n+1}$, so that the construction can continue. Since we are basically just implementing the Galvin--Glazer argument along our tree, it is easy to see that for every branch $f$ of the tree $q$ it will be the case that $\fu(f[\{|s_p|+n\big|n<\omega\}])\subseteq A$. Note that this also implies that 
\begin{equation*}
q\Vdash``\fu(\mathring{X}\setminus\ran(\check{s_p}))\subseteq\check{A}"
\end{equation*}
(since for every finite subset $a\subseteq\omega$, there is an extension $r\leq q$ deciding 
that the generic set $\mathring{X}$ coincides with some ground-model branch $f$ of $q$ up to 
the $|s_q|+\max(a)$-th element), and we are done.
\end{proof}

We will now attempt to build an iteration of forcing notions of the form $\mathbb L(u)$, for various $u$, with countable support (if we do it with finite support, then we will be adding cofinally many Cohen reals, and so by the results in Section 2 we would already know that there will be union ultrafilters). Now, if we start with a model of $\ch$ and proceed to iterate proper forcing notions with countable support, the iteration must be of length (at least) $\omega_2$ if we want the generic extension to satisfy $\neg\ch$. Furthermore, we need to make sure that the iteration does not add any Cohen reals, lest we end up with a model where $\covm=\mathfrak c$, which is not what we want. We proceed to lay out some conditions under which forcing notions of the form $\mathbb L(u)$, and their iterations, do not add Cohen reals.

\begin{definition}
A forcing notion $\mathbb P$ is said to satisfy the {\it Laver property} if whenever $g:\omega\longrightarrow\omega$ (in the ground model), $p$ is a condition, and $\mathring{f}$ is a $\mathbb P$-name such that
\begin{equation*}
 p\Vdash``\mathring{f}:\omega\longrightarrow\omega\text{ and }\mathring{f}\leq\check{g}",
\end{equation*} 
there is $F:\omega\longrightarrow[\omega]^{<\omega}$ and $q\leq p$ such that for every $n<\omega$, 
 $|F(n)|\leq2^n$ and $q\Vdash``\mathring{f}(\check{n})\in\check{F}(\check{n})"$.
\end{definition}

The Laver property is important because of two reasons. The first is that it is preserved under countable support iterations~\cite[Theorem 6.3.34]{librogris}, and the second is that, whenever $\mathbb P$ has the Laver property, it does not add any Cohen reals~\cite[Lemma 7.2.3]{librogris} (or see~\cite[2.10D]{shelahproperimproper}). This means that, if we iterate proper forcing notions satisfying the Laver property with countable support, then the full iteration itself will not have added any Cohen reals. Therefore, if our ground model satisfies $\ch$, in the generic extension it will be the case that $\covm=\omega_1<\omega_2=\mathfrak c$. We do not know exactly how to characterize the ultrafilters $u$ such that $\mathbb L(u)$ has the Laver property. We will see, however, that the answer to this will be affirmative in case $u$ is a stable ordered union ultrafilter. To see this, it will be convenient to recall a characterization of stable ordered union ultrafilters in terms of games.

\begin{definition}
Given an ultrafilter $u\in\gamma\fin$, we define a game $\mathscr G(u)$ as follows: in the $n$-th inning, player I plays a set $A_n\in u$ and then player II responds with an element $x_n\in A_n$. After $\omega$ moves, we collect player II's moves into a family $X=\{x_n\big|n<\omega\}$, and player II wins if and only if $\fu(X)\in u$.
\end{definition}

The definition of the above game, as well as the characterization of stable ordered union ultrafilters in terms of it, is part of a still unpublished joint work of the author with Peter Krautzberger and David Chodounsk\'y. The reader interested in seeing a written proof of the two claims that follow can consult~\cite[Section 4.2, pp. 112--116]{myphdthesis}. The first claim is that, if $u$ is nonprincipal, then it is impossible for player II to have a winning strategy (this can be proved with the usual argument where players I and II play two games in parallel, in such a way that II cannot possibly win both games). The second claim is the following theorem.

\begin{theorem}[\cite{myphdthesis}, Theorem 4.17]\label{gamechar}
Let $u\in\gamma\fin$ be an idempotent ultrafilter. Then, $u$ is stable ordered union if and only if player I does not have a winning strategy in the game $\mathscr G(u)$.
\end{theorem}

Using the above result, we are ready to prove that the forcing $\mathbb L(u)$ has the Laver property for stable ordered union $u$.

\begin{theorem}\label{laverprop}
If $u\in\gamma\fin$ is a stable ordered union ultrafilter, then $\mathbb L(u)$ satisfies the Laver property.
\end{theorem}

\begin{proof}
Let $p\in\mathbb L(u)$, $g:\omega\longrightarrow\omega$ and $\mathring{f}\in V^{\mathbb L(u)}$
be such that 
\begin{equation*}
p\Vdash``\mathring{f}:\check{\omega}\longrightarrow\check{\omega}\text{ and }\mathring{f}\leq\check{g}"
\end{equation*}
We will recursively construct an extension $q\leq p$ that will satisfy that 
\begin{equation*}
q\Vdash``(\forall n<\check{\omega})(\mathring{f}(n)\in\check{F}(n))"
\end{equation*}
for some ground-model slalom $F:\omega\longrightarrow[\omega]^{<\omega}$ satisfying $(\forall n<\omega)(|F(n)|\leq 2^n)$. We first let $h(n)$ be the number of finite sequences of natural numbers $\langle m_1,\ldots,m_k\rangle$ that satisfy $\sum_{i=1}^k \lfloor\log_2(m_i)\rfloor=n$, and we pick and fix any increasing sequence $\langle k_n\big|n<\omega\rangle$ satisfying that $2^{k_n}\geq 2^{n+1} h(n+1)$. We moreover use the fact that for every Laver condition, there is a natural order-preserving bijection between $\omega^{<\omega}$ and the nodes of the condition above the stem.
 
We now define $q$ by induction on the nodes. That is, if we have already decided that a certain $t\in p$ will belong to $q$, we will show how to pick the set of immediate succesors $\suc_{q}(t)$. For this, we will assume that not only have we decided that $t\in q$, but we have also decided which will be the sequence $\langle m_1,\ldots,m_k\rangle$ associated to $t$ under the aforementioned order-preserving bijection (between $\omega^{<\omega}$ and the nodes of $q$ above the stem) and we also assume that we have picked an auxiliary condition $p_t\leq^0 p\upharpoonright t$ (here $p\upharpoonright t$ denotes the condition $\{s\in p\big|s\text{ is comparable with }t\}$) which decides the value of $f\upharpoonright k_n$, where $n=\sum_{i=1}^k \lfloor\log_2(m_i)\rfloor$.
 
Now we play the game $\mathscr G(u)$. The first thing to do is shrink, if necessary, the set $\suc_p(t)$ to a set of the form $\fu(Y)$, so that it is closed under finite unions. This way, at the end of the game we will be able to collect player II's moves $X=\{x_n\big|n<\omega\}$ and we will define $\suc_{q}(t)=\fu(X)$. Player I will adhere to the following strategy. First extend, for each $s\in\suc_p(t)=\fu(Y)$, the condition $p_t\upharpoonright s$ to some condition $p_s^0$ with the same stem deciding the value of $\mathring{f}\upharpoonright k_{n+1}$ to be a certain $f_s^0$. The hypothesis that $p\Vdash``\mathring{f}\leq\check{g}"$ implies that there are only finitely many possible $f_s^0$, so there is a set $A_0\in u$ such that all $p_s^0$ for $s\in A_1$ decide $\mathring{f}\upharpoonright k_{n+1}$ to be the same $f_0$. Player I starts by playing this set, and waits for player II to play some $x_0\in A_0$. The auxiliary condition associated to $x_0$ in order to continue with the induction later on, will be $p_{x_0}^0$. We now extend, for each $s\in\fu(Y)\setminus\{x_0\}$, the condition $p_s^0$ to some further condition $p_s^1$ with the same stem which decides the value of $\mathring{f}\upharpoonright k_{n+2}$ to be some $f_s^1$. Now (and here is the interesting twist), to each such $s$ we associate the pair $\langle f_s^1,f_{s\bigtriangleup x_0}^1\rangle$, and since there are only finitely many possibilities for such a pair, there exists a set $A_1\in u$ such that for all $s\in A_1$ the aforementioned pair is constantly some fixed pair $\langle f_1,f_2\rangle$. Then we let player I play the set $A_1$ and wait for player II's response $x_1\in A_1$. We will let the auxiliary conditions associated to $x_1$ and $x_0\bigtriangleup x_1$ be $p_{x_1}^1$ and $p_{x_0\bigtriangleup x_1}^1$, respectively.
 
In general, if we are about to play the $m$-th inning of the game $\mathscr G(u)$, we assume that we know $\vec{x}=\langle x_i\big|i<m\rangle$ and the auxiliary conditions associated to each $x\in\fs(\vec{x})$, which decide the value of $\mathring{f}\upharpoonright k_{n+\max\{i<m|x_i\subseteq x\}+1}$. We now extend, for each $s\in\fu(Y)\setminus\fu(\vec{s})$, the condition $p_s^{m-1}$ to some pure extension $p_s^m$ which decides the value of $\mathring{f}\upharpoonright k_{n+m+1}$ to be a certain $f_s^m$. Since there are only finitely many possibilities for the vector 
\begin{equation*}
\langle f_s^m\rangle\frown\langle f_{s\bigtriangleup x}^m\big|x\in\fu(\vec{x})\rangle,
\end{equation*}
then there exists an $A_m\in u$ such that for all $s\in A_m$, the aforementioned vector is some fixed $\langle f_{2^{m-1}+1},\ldots,f_{2^m}\rangle$. We let player I play the set $A_m$ and wait for player II's response $x_m\in A_m$, and we establish that the auxiliary condition associated to $x_m$ is $p_{x_m}^m$ and the one associated to $x_m\bigtriangleup x$ will be $p_{x_m\bigtriangleup x}^m$, for each $x\in\fu(\vec{x})$.
 
In the end, since the described strategy cannot be winning, there is a possibility for player II to have won the game, i.e., $\fu(X)\in u$. For each $x\in\fu(X)$, we let the sequence associated to $t\frown\langle x\rangle$ (for the order-preserving bijection with $[\omega]^{<\omega}$) be $\langle m_1,\ldots,m_k\rangle\frown\langle\sum_{x_i\subseteq x}2^i\rangle$, and the induction can continue. It is important to note that, for every $x\in\fu(X)$ and every $j\leq\max\{i<\omega|x_i\in\supp_X(x)\}$, the auxiliary condition $p_{t\frown\langle x\rangle}$ forces the value of $f\upharpoonright k_{n+j+1}$ to agree with some entry of the vector $\langle f_{2^{t-1}+1},\ldots,f_{2^t}\rangle$ (where $t=\lfloor\log_2(j)\rfloor$) which was chosen during the $t$-th run of the 
game $\mathscr G(u)$.

This way we get our condition $q\leq p$ (in fact, $q$ and $p$ have the same stem). It is straightforward to check that, given any $n<\omega$, if $k_{i-1}\leq n<k_i$ (with the convention that $k_{-1}=0$) then $q\Vdash``\mathring{f}(\check{n})\in\check{F(n)}"$, where $F(n)$ is the collection of all entries from the vectors $\langle f_{2^i+1},\ldots,f_{2^{i+1}}\rangle$ obtained when doing the induction over a node $t\in q$ whose associated sequence (under the bijection with $[\omega]^{<\omega}$) is some $\langle m_1,\ldots,m_k\rangle$ satisfying $\sum_{j=1}^k\lfloor\log_2(m_k)\rfloor=i$. Since there are only $h(i)$ many such sequences, it follows that $|F(n)|\leq 2^i h(i)\leq 2^{k_{i-1}}\leq 2^n$.
\end{proof}

Now we know that iterating forcings of the form $\mathbb L(u)$, where each of the ultrafilters $u$ is stable ordered union, will yield a generic extension with a small $\covm$. In order to perform such an iteration in a way that the generic extension itself has stable ordered union ultrafilters, we will at some point need the following result of Eisworth (the way we state the result below is a very particular case of Eisworth's generic existence result from~\cite{eisworth-union}).

\begin{lemma}[See~\cite{eisworth-union}, Theorem~9]\label{extendunion}
If we assume $\ch$, then any filter generated by countably many sets of the form $\fs(X)$, with $X$ a block sequence, can be extended to a stable ordered union ultrafilter.
\end{lemma}

With the above tools, we are finally ready to establish the main result of this section.

\begin{theorem}
Suppose that $\mathbf{V}\vDash\ch$. Then there is a countable support iteration $\langle\mathbb P_\alpha,\mathring{\mathbb Q_\alpha}\big|\alpha<\omega_2\rangle$, satisfying
\begin{equation*}
\mathbb P_\alpha\Vdash``\mathring{\mathbb Q_\alpha}=\mathbb L(u)\text{ for some stable ordered union ultrafilter }u"
\end{equation*}
for every $\alpha<\omega_2$, and such that the corresponding generic extension satisfies that $\covm=\omega_1<\omega_2=\mathfrak c$ and that there exists a stable ordered union ultrafilter.
\end{theorem}

\begin{proof}
We recursively define iterands $\mathring{\mathbb Q_\alpha}$ by specifying the (name of the) stable ordered union ultrafilter $\mathring{u_\alpha}\in\mathbf{V}^{\mathbb P_\alpha}$ and stipulating that $\mathbb P_\alpha\Vdash``\mathring{\mathbb Q_\alpha}=\mathbb L(\mathring{u_\alpha})"$. We furthermore let $\mathring{X_\alpha}$ be the name of the generic block sequence added by the $\alpha$-th iterand $\mathring{Q_\alpha}$, which is (forced to be) $\mathbb L(\mathring{u_\alpha})$, to the generic extension obtained after forcing with $\mathbb P_{\alpha}$ (this is originally a $\mathbb P_\alpha$-name, which later on is identified with the corresponding $\mathbb P_\eta$-name, for $\alpha<\eta$, in the usual way). As we go along our construction, we will ensure that whenever $\xi<\alpha<\omega_2$, it is the case that $\mathbb P_\alpha\Vdash``\mathring{u_\xi}\subseteq\mathring{u_\alpha}"$, and also that $\mathbb P_\alpha\Vdash``\fu(\mathring{X_\alpha}\setminus\check{F})\in\mathring{u_{\alpha+1}}"$, for every finite $F$.

Thus, suppose that $\mathbb P_\alpha$ has already been defined, along with the names $\mathring{u_\xi}$ and $\mathring{X_\xi}$, for all $\xi<\alpha$. There are three cases to consider:
\begin{description}
\item[$\alpha$ is a successor ordinal] Say that $\alpha=\xi+1$; in this case we have that $\mathbb P_\alpha=\mathbb P_\xi\star\mathbb L(\mathring{u_\xi})$. By Lemma~\ref{laverdiagonalizes}, in the generic extension obtained after forcing with $\mathbb P_\alpha$, the filter generated by the family $\{\fu(X_\alpha\setminus F)\big|F\text{ is finite}\}$ will extend the ultrafilter $u_\xi$, and consequently it will also extend all of the $u_\eta$ for $\eta<\xi$ (since by inductive hypothesis, $u_\eta\subseteq u_\xi$). We just let $\mathring{u_\alpha}$ be an arbitrary name for a stable ordered union ultrafilter satisfying
\begin{equation*}
\mathbb P_\alpha\Vdash``\mathring{u_\alpha}\text{ extends }\{\fu(\mathring{X_\alpha}\setminus F)\big|F\text{ is finite}\},
\end{equation*}
guaranteed to exist by Lemma~\ref{extendunion}; $\mathring{u_\alpha}$ is as required because it (is forced to) extends all of the $\mathring{u_\eta}$ for $\eta<\alpha$ and it also contains all of the $\fu(\mathring{X_\alpha}\setminus\check{F})$ by definition.

\item[$\alpha$ is a limit ordinal with $\cf(\alpha)=\omega$]
Pick an increasing cofinal sequence $\langle\alpha_n\big|n<\omega\rangle$ converging to $\alpha$. Notice that, if $\xi<\alpha$ and $\mathring{A}$ is such that $\mathbb P_\alpha\Vdash``\mathring{A}\in\mathring{u_\xi}"$, then there will be an $n<\omega$ and a finite $F$ such that $\mathbb P_\alpha\Vdash``\fu(\mathring{X_{\alpha_n}}\setminus\check{F})\subseteq\mathring{A}"$. To see this, just let $n$ be such that $\xi<\alpha_n$, and use the fact that $\mathbb P_\alpha\Vdash``\mathring{u_\xi}\subseteq\mathring{u_{\alpha_n}}"$ along with Lemma~\ref{laverdiagonalizes}. Hence, working in the generic extension obtained after forcing with $\mathbb P_\alpha$, if it were the case that $\mathscr F=\{\fu(X_{\alpha_n}\setminus F)\big|n<\omega\text{ and }F\text{ is finite}\}$ generated a filter, then every ultrafilter extending $\mathscr F$ would extend each of the $u_\xi$ ($\xi<\alpha$) as well. Fortunately, it is straightforward to check that $\mathscr F$ indeed generates a filter, as (since each such $\fu(X_{\alpha_n}\setminus F)$ belongs to $u_{\alpha_n+1}$) $\mathscr F\subseteq\bigcup_{\xi<\alpha}u_\xi$ and the latter is a filter (it being an increasing union of filters). Since $\mathscr F$ is a countable family of $\fu$-sets, a simple application of Lemma~\ref{extendunion} guarantees the existence of a (name for a) stable ordered union ultrafilter $\mathring{u_\alpha}$ extending $\mathscr F$, and therefore satisfying $\mathbb P_\alpha\Vdash``\mathring{u_\xi}\subseteq\mathring{u_\alpha}"$ for all $\xi<\alpha$, as desired.

\item[$\alpha$ is a limit ordinal with $\cf(\alpha)>\omega$]
Working in the generic extension obtained after forcing with $\mathbb P_\alpha$, define $u_\alpha=\bigcup_{\xi<\alpha}u_\xi$, which is by definition a filter with a basis of $\fu$-sets. Since $\alpha$ has uncountable cofinality, every element $A\in\wp(\fin)$ occuring in our extension actually occurs already at stage $\xi$ for some $\xi<\alpha$, at which point we know that either $A$ or $\fin\setminus A$ belongs to $u_\xi\subseteq u_\alpha$, which shows that the filter $u_\alpha$ is indeed maximal. Furthermore, every countable family $\{A_n\big|n<\omega\}\subseteq u_\alpha$ must also occur at some stage $\xi<\alpha$, and therefore (since $u_\xi$ is stable ordered union and every $A_n\in u_\xi$) there is a block sequence $X$, with $\fu(X)\in u_\xi\subseteq u_\alpha$, such that for all $n<\omega$ there is a finite $F$ with $\fu(X\setminus F)\subseteq A_n$. This shows that $u_\alpha$ is a stable ordered union ultrafilter in our generic extension, so we just let $\mathring{u_\alpha}$ be the corresponding $\mathbb P_\alpha$-name, and we are done. 
\end{description}
The above shows that we can actually perform the recursive construction as desired. In the end, letting $u=\bigcup_{\alpha<\omega_2}u_\alpha$, we notice that, since $\omega_2$ is a limit ordinal of uncountable cofinality, the filter $u$ must actually be a stable ordered union ultrafilter, by the exact same reasoning as in the third case of the construction above. Furthermore, by Theorem~\ref{laverprop}, we get that $\covm=\omega_1$ in the generic extension, since $\mathbb P_{\omega_2}$ does not add Cohen reals.
\end{proof}

\section{A single-step iteration}

We will now exhibit a (single-step) forcing notion that adds a generic stable ordered union ultrafilter; this forcing notion has been considered by Blass~\cite{blassunion} and taken advantage of by Eisworth~\cite{eisworth-union}; it has been used afterwards by other authors as well. By ensuring that appropriate conditions hold in the ground model, we will be able to get a generic extension satisfying $\covm<\mathfrak c$, thus yielding yet another model of $\zfc$ that has stable ordered union ultrafilters and a small $\covm$.

\begin{definition}
We will let $\fin^{[\infty]}$ denote the set of all infinite block sequences. This set will be considered as a forcing notion, equipped with the preorder given by almost condensation. In other words, $X\leq Y$ if and only if there exists a finite $F$ such that $X\setminus F\subseteq\fu(Y)$.
\end{definition}

It is hard not to see that this forcing notion is $\sigma$-closed (in fact, given a countably infinite decreasing sequence in $\fin$, a very explicit method to find a common lower bound for this sequence can be found in~\cite[Lemma 3.2]{conmichael-diamond-union}). Therefore this notion does not add any new reals when one passes to a generic extension. The generic filter in such an extension is a stable ordered union ultrafilter.

Thus, whenever one starts with a ground model $\mathbf{V}$ and forces with $\fin^{[\infty]}$, the generic extension $\mathbf{V}[G]$ will contain the stable ordered union ultrafilter $G$. The question now is under which circumstances does the ground model $\mathbf{V}$ ensure that $\mathbf{V}[G]\vDash``\covm<\mathfrak c"$. This section is devoted to exhibiting one such model. The following definition is relevant to this.

\begin{definition}
The cardinal characteristic $\mathfrak h_\fin$ is defined to be the least cardinality of a family of dense sets in $\fin^{[\infty]}$ with a non-dense intersection.
\end{definition}

Thus, $\mathfrak h_\fin$ is just what Balcar, Doucha, and Hru\v s\'ak~\cite{balcar-doucha-hrusak} call the \textit{height} of the forcing notion $\fin^{[\infty]}$. As a matter of fact, the following are all equivalent ways of defining the same cardinal characteristic (substituting $\fin^{[\infty]}$ for any other forcing notion, we still have equivalent definitions for the height of said forcing notion, see, e.g.,~\cite[p. 158]{jech}).

\begin{itemize}
\item $\mathfrak h_\fin$ is the least cardinality of a family of maximal antichains in $\fin^{[\infty]}$ with no common refinement,
\item $\mathfrak h_\fin$ is the distributivity of the Boolean algebra of regular open sets of $\fin^{[\infty]}$,
\item $\mathfrak h_\fin$ is the shortest length of a sequence of ground-model objects (equivalently, ordinals) added by the forcing $\fin^{[\infty]}$.
\end{itemize}

Notice that any two block sequences $X,Y$ are isomorphic, in the following sense: if we let $\langle x_n\big|n<\omega\rangle$ and $\langle y_n\big|n<\omega\rangle$ be increasing enumerations of $X$ and $Y$, respectively (that is, $\max(x_n)<\min(x_{n+1})$, and similarly for the $y_n$), then the function $\varphi:\fu(X)\longrightarrow\fu(Y)$ given by $\varphi(\bigcup_{n\in F}x_n)=\bigcup_{n\in F}y_n$ maps bijectively $\fu(X)$ onto $\fu(Y)$, and furthermore for every condensation $Z$ of $X$, the block sequence $\varphi[Z]$ will be a condensation of $Y$. This shows that the forcing $\fin^{[\infty]}$ is homogeneous in the very strong sense that, for every two conditions $X,Y\in\fin^{[\infty]}$, the downward closure of $X$ is isomorphic to the downward closure of $Y$ (in particular, letting $X=\mathbb 1=\{\{n\}\big|n<\omega\}$, we get that for every condition $X$, the downward closure of $X$ is isomorphic to the whole forcing notion). This implies that, furthermore, we can also equivalently define the cardinal characteristic $\mathfrak h_\fin$ as follows (arguing exactly as in \cite[p. 426]{blasscardinv}):

\begin{itemize}
\item $\mathfrak h_\fin$ is the least cardinality of a family of dense sets of $\fin^{[\infty]}$ that has an empty intersection.
\end{itemize}

Moreover, since the forcing notion $\fin^{[\infty]}$ is an atomless $\sigma$-closed forcing notion of size $\mathfrak c$, by~\cite[Theorem 2.1]{balcar-doucha-hrusak} it has the \textit{base tree property}, meaning that there is a dense subset $\mathscr T\subseteq\fin^{[\infty]}$ which forms a tree (under the order inherited from $\fin^{[\infty]}$) of height $\mathfrak h_\fin$ such that every level of the tree is a maximal antichain, and every node of the tree has $\mathfrak c$-many successors. This in turn implies that forcing with $\fin^{[\infty]}$ is forcing equivalent to forcing with the tree $\mathscr T$ (since $\mathscr T$ is dense in $\fin^{[\infty]}$), and therefore (since every node of the tree has $\mathfrak c$-many successors) it generically adjoins a surjection $:\mathfrak h_\fin\longrightarrow\mathfrak c$ (while not adjoining any reals nor any shorter sequences of ground-model objects).

Given the facts stated in the previous paragraph, if we start with any ground model $\mathbf{V}$ satisfying $\kappa=\covm<\mathfrak h_\fin=\lambda$, then in the generic extension it will be the case that $\kappa=\covm<\mathfrak c=\lambda$ (since $\fin^{[\infty]}$ is $\sigma$-closed, it does not add any reals and so any ground-model witness for $\covm$ will continue to be one in the generic extension). Thus our main task now is to show that there exists a model of $\zfc$ satisfying $\covm<\mathfrak h_\fin$. To see this, we will utilize the following forcing notion. This forcing notion has been considered previously by Garc\'{\i}a-\'Avila~\cite{luzmaria-tesis,luzmaria-forcinghindman}, who denoted it by $\mathbb P_\fin$.

\begin{definition}
The Matet--Mathias forcing notion $\mm$ will be defined as follows: conditions in $\mm$ are pairs $(\vec{x},X)$ such that $\vec{x}=\langle x_0,\ldots,x_n\rangle$ is a finite block sequence in $\fin$, and $X\in\fin^{[\infty]}$; we stipulate that $(\langle x_0,\ldots,x_n\rangle,X)\leq(\langle y_0,\ldots,y_m\rangle,Y)$ if and only if $m\leq n$, $(\forall k\leq m)(x_k=y_k)$, and $\{x_{m-1},\ldots,x_n\}\cup X\sqsubseteq Y$ (in particular this implies that $\{x_{m-1},\ldots,x_n\}\cup X$ is itself a block sequence). The finite block sequence $\vec{x}$ that constitutes the first coordinate of the condition $(\vec{x},X)\in\mm$ will be called the {\it stem} of the condition.
\end{definition}

This forcing notion resembles Mathias forcing in that the stem of a condition $(\vec{x},X)$ is a finite sequence of elements of $\fin$, although it also resembles Matet forcing in that when extending conditions, one is allowed to take finite unions over the elements of the side condition $X$. However, the forcing notion $\mm$ is not forcing equivalent to either Matet or to Mathias forcing (see~\cite[Sections 4.1, 5.2]{luzmaria-tesis} or~\cite[Corollaries 4.5, 5.19]{luzmaria-forcinghindman}). Standard considerations show that, if $u$ is an ordered union ultrafilter, and if we define the guided Matet-Mathias forcing $\mm(u)$ by letting conditions be pairs $(\vec{x},X)$ such that $\fu(X)\in u$, then the mapping $(\vec{x},X)\longmapsto\langle X,(\vec{x},\check{X})\rangle$ constitutes a dense embedding of $\mm$ into the forcing notion $\fin^{[\infty]}\star\mm(\mathring{u})$, where $\mathring{u}$ is the (name for the) generic stable ordered union ultrafilter added by $\fin^{[\infty]}$.

Let $\mathring{Z}$ be an $\mm$-name for the union of all of the $\ran(\vec{z})$, where $\vec{z}$ ranges over the stems of all conditions in the generic filter. We say that $\mathring{Z}$ is the name of the {\it generic block sequence} added by $\mm$. Notice that any condition $(\vec{x},X)$ forces that $\vec{x}$ is an initial segment of $Z$, and that the generic block sequence $Z$ is an almost condensation of $X$ (more concretely, $(\vec{x},X)\Vdash``\mathring{Z}\setminus\ran(\vec{x})\sqsubseteq\check{X}"$). By a straightforward density argument, we see that in the generic extension $Z$ belongs to every dense subset of $\fin^{[\infty]}$ from the ground model. In particular, the forcing notion $\mm$ destroys all ground-model witnesses for $\mathfrak h_\fin$. It is known that this forcing notion satisfies various highly desirable properties that make it ideal for iterating with countable support.

\begin{lemma}\label{properbounding}
The forcing notion $\mm$ satisfies the following properties:
\begin{itemize}
\item The pure decision property, stating that finitary decisions can be made without modifying the stem of a condition. In other words, if $F$ is a (ground model) finite set, $\mathring{a}$ is an $\mm$-name, and $(\vec{x},X)\in\mm$ is a condition satisfying $(\vec{x},X)\Vdash``\mathring{a}\in\check{F}"$, then there exists a condensation $Y\sqsubseteq X$ and an element $b\in F$ such that $(\vec{x},Y)\Vdash``\mathring{a}=\check{b}"$. This is proved in~\cite[Theorem 4.25]{luzmaria-tesis} or~\cite[Theorem 4.21]{luzmaria-forcinghindman}.
\item Properness. This is proved in~\cite[Theorem 4.23]{luzmaria-tesis} or~\cite[Theorem 4.19]{luzmaria-forcinghindman}.
\item The Laver property. This is proved in~\cite[Theorem 4.26]{luzmaria-forcinghindman}.
\end{itemize}
\end{lemma}

This gets us in a position of obtaining a model for $\covm<\mathfrak h_\fin$.

\begin{theorem}\label{largehmodel}
There exists a model of $\zfc$ that satisfies $\covm=\omega_1<\omega_2=\mathfrak h_\fin$
\end{theorem}

\begin{proof}
Let $\mathbf{V}$ be a model of $\ch$, and let us iterate the Matet--Mathias forcing notion $\mm$, $\omega_2$-many times, with countable support. By Lemma~\ref{properbounding}, along with the well-known preservation theorems referenced to in the previous section, the final result of the iteration $\mathbb P_{\omega_2}$ will be a proper forcing notion satisfying the Laver property; in particular, no Cohen reals are added and therefore the generic extension satisfies $\covm=\omega_1$. Furthermore, a standard reflection argument shows that any family of $\omega_1$-many dense subsets of $\fin^{[\omega]}$ in the generic extension will already belong to some intermediate model, and thus will be killed by the subsequent iteration of $\mm$. Therefore the generic extension also satisfies $\mathfrak h_\fin=\omega_2$.
\end{proof}

\begin{corollary}
It is consistent with $\zfc$ that forcing with $\fin^{[\infty]}$ (which generically adds a stable ordered union ultrafilter) produces a generic extension satisfying $\covm<\mathfrak c$.
\end{corollary}

\begin{proof}
Take the model from Theorem~\ref{largehmodel}. Since $\fin^{[\infty]}$ is $\mathfrak h_\fin$-closed and $\covm<\mathfrak h_\fin$, the value of $\covm$ will not be changed by this forcing. Furthermore, the generic extension by $\fin^{[\infty]}$ contains the same reals as the ground model, and no new $\omega_1$-sequences of reals, and therefore in the generic extension we have $\mathfrak c=\omega_2$.
\end{proof}

In the previous corollary, we force over a ground model satisfying $\mathfrak h_\fin=\omega_2$, so it is natural to ask what is the value of $\mathfrak h_\fin$ in the extension by $\fin^{[\infty]}$. We conjecture that, in fact, if our ground model is that of Theorem~\ref{largehmodel}, then $\mathfrak h_\fin$ becomes $\omega_1$ in the generic extension. Equivalently, our conjecture can be stated as saying that in the model from Theorem~\ref{largehmodel}, the value of $\mathfrak h_\fin^2$ --the distributivity number of the forcing notion $\fin^{[\infty]}\times\fin^{[\infty]}$ that adds two mutually generic stable ordered union ultrafilters-- is equal to $\omega_1$. This would be analogous to the result of Shelah and Spinas~\cite{she-spi-h-h2} that $\mathfrak h_2=\omega_1<\omega_2=\mathfrak h$ in Mathias's model.

\bigskip

We conclude this article by stating a few remarks about the values of other cardinal characteristics of the continuum in the models that we have obtained. The family of models considered in Section 2 satisfies that $\mathfrak d<\mathfrak c$, although I do not know if this is unavoidable. However, since Cohen reals are splitting, it follows that all of the models from Section 2 satisfy that $\mathfrak r<\mathfrak c$ (and consequently also $\mathfrak b<\mathfrak c$. The model from Section 3, on the other hand, satisfies that $\mathfrak b=\mathfrak c$, which implies that $\mathfrak d=\mathfrak r=\mathfrak c$ as well, in sharp contrast with the models from Section 2.

When comparing the model from this section with that of Section 3, we were not able to find any cardinal characteristic whose values in each of these models differ. To analyze the model from this section, notice that, after iterating the forcing $\mm$ but before forcing with $\fin^{[\infty]}$ we have that $\mathfrak h_\fin=\mathfrak c$; and this invariant is provably $\leq\!\mathfrak h$ (in fact, it is $\leq\!\mathfrak h_2$, since the Rudin-Keisler images of the generic stable ordered union ultrafilter under the mappings $\max$ and $\min$ give rise to two mutually generic selective ultrafilters). Hence we will have that $\mathfrak b=\mathfrak c$; at this point we proceed to force with $\fin^{[\infty]}$, which is $\sigma$-closed and hence does not add any new reals, nor does it add any new $\omega_1$-sequence of reals, and therefore in the final model we still have that $\mathfrak b=\mathfrak c$ (and consequently $\mathfrak d=\mathfrak r=\mathfrak c$ as well). Looking at Cicho\'n's diagram, the only two cardinal characteristics whose value is not determined by the above considerations are $\nonn$ and $\covn$. For $\covn$, notice that it will have value $\omega_1$ both in the ultraLaver extension from Section 3, and in the model obtained by iterating $\mm$; this is due to the fact that both iterations satisfy the Laver property, and hence they do not add any Random reals by~\cite[Lemma 7.2.3]{librogris}, hence the ground-model null sets cover the new real line. Since $\fin^{[\infty]}$ does not add any new reals nor any new Borel sets of reals, forcing with it will not change the fact that $\covn=\omega_1$; therefore both the model from Section 3 and the one from the present section satisfy that $\covn=\omega_1$. On the other hand, the value of $\nonn$ will be $\mathfrak c$ in both of these models. To see this, for the model of Section 3, notice that the generic ultraLaver real is not split by any ground-model real (it in fact reaps them all), and therefore we have that ultraLaver forcing destroys any ground-model witnesses for $\mathfrak s$; so in the final extension we will have that $\mathfrak s=\mathfrak c$. Now, for the model of the present section, since we argued that the extension obtained after iterating $\mm$ satisfies $\mathfrak h=\mathfrak c$, then we will also have that $\mathfrak s=\mathfrak c$ (since the inequality $\mathfrak h\leq\mathfrak s$ is provable in $\zfc$); this fact will not be changed by forcing with a $\sigma$-closed forcing notion with distributivity number equal to $\mathfrak s$. Hence both the models from Section 3 and from the present section satisfy $\mathfrak s=\mathfrak c$, which implies that they also satisfy $\nonn=\mathfrak c$ as the inequality $\mathfrak s\leq\nonn$ is provable in $\zfc$. Thus, the values of all cardinal characteristics from Cicho\'n's diagram, as well as a few others, have been proven to be equal in both of these models.

\begin{question}
Is there a cardinal characteristic of the continuum whose value in the model from Section 3 (iteration of appropriate ultraLaver forcings) is different from its value in the model from this section (iteration of the Matet--Mathias forcing followed by a single-step forcing with $\fin^{[\infty]}$)?
\end{question}

Finally, we mention two possible topics for future research. It is worth noting that, in the constructions presented in this article, we essentially have no control over the value of the cardinal characteristic $\mathfrak d$ (since we have to add dominating reals in every step of every iteration considered), so it would be interesting to see whether the techniques of Blass and Shelah~\cite{blass-shelah-firstmatrix} can be adapted to this context, and one can thus obtain models with union ultrafilters and $\covm<\mathfrak c$, but with more variability on the value of $\mathfrak d$. Another fact worth noticing is that, in all of the models constructed in this article, our stable ordered union ultrafilters cannot have small (i.e., of size $\omega_1$) bases, whereas, e.g., the stable ordered union ultrafilters from~\cite{conmichael-diamond-union} do have small bases; some investigation about the minimal size of a stable ordered union ultrafilter basis in models that have such ultrafilters is almost certain to generate interesting results (the author is very thankful to the anonymous referee for pointing out these two suggestions for further research).

{\bf Acknowledgments:} In its origins, the problem of obtaining models of $\zfc$ with union ultrafilters and a small $\covm$ was addressed by the author in his Ph.D. thesis~\cite{myphdthesis}, written at York University under the supervision of Juris Stepr\=ans. In fact, the main result from Section 3 of this article is Theorem~4.26 on said thesis. Meanwhile, the results from Section 2 here vastly generalize other results from there; more concretely, Theorems~4.20 and~4.23 on~\cite{myphdthesis} are both subsumed by our Corollary~\ref{fsmodel}. The material from Section 4 is completely new. The author is grateful to Juris Stepr\=ans for useful discussions during his (the author's) time as a Ph.D. student, to Andreas Blass for other useful discussions afterwards, and to David Chodounsk\'y for pointing out reference~\cite{balcar-doucha-hrusak} (which rescued me from a fair amount of re-inventing the wheel).

\end{document}